\documentclass[11pt, reqno]{amsart}
\usepackage{graphicx,psfrag}
\usepackage[psamsfonts]{amssymb}
\usepackage{pdfsync}
\usepackage{amscd}
\usepackage{amsmath}
\usepackage{amsxtra}
\usepackage{mathrsfs}
\usepackage{stmaryrd}
\usepackage[small,nohug,heads=littlevee]{diagrams}
\usepackage{comment}
\usepackage{enumerate}
\usepackage{hyperref}
\usepackage{xfrac}
\diagramstyle[labelstyle=\scriptstyle]

\usepackage{mathpple}

\usepackage[width=5.8in, height=8.5in, bottom=1.3in, centering]{geometry}

\theoremstyle{plain}
    \newtheorem{theorem}{Theorem}[section]
    \newtheorem{lemma}[theorem]{Lemma}
    \newtheorem{corollary}[theorem]{Corollary}
    \newtheorem{proposition}[theorem]{Proposition}

\theoremstyle{definition}
    \newtheorem{definition}[theorem]{Definition}

    \newtheorem*{thank}{Acknowledgements}
\theoremstyle{remark}
    \newtheorem{remark}[theorem]{Remark}

    \newtheorem*{digression}{Digression}

\numberwithin{equation}{section}

\newcommand{\NN}{\mathbb{N}}

\newcommand{\Conn}{\mathfrak{Conn}}

\newcommand{\real}{\mathbb{R}}
\newcommand{\complex}{\mathbb{C}}

\newcommand{\A}{\mathcal{A}}

\newcommand{\F}{\mathcal{F}}

\newcommand{\AOO}{\A^{0,0}}

\newcommand{\AOD}{\A^{0,\bullet}}
\newcommand{\Homsheaf}{\mathscr{H}om}

\newcommand{\PA}{\mathcal{P}_{A}}

\newcommand{\gext}{\operatorname{Ext}}

\newcommand{\Hom}{\operatorname{Hom}}
\newcommand{\End}{\operatorname{End}}
\newcommand{\Id}{\operatorname{Id}}
\newcommand{\dbar}{\overline{\partial}}
\newcommand{\expmap}{\operatorname{exp}}
\newcommand{\Expmap}{\operatorname{Exp}}
\newcommand{\normal}{N}
\newcommand{\conormal}{N^\vee}

\newcommand{\Dcoh}{\mathcal{D}^b_{coh}}

\newcommand{\Cinf}{$C^\infty$}

\newcommand{\aaa}{\mathfrak{a}}
\newcommand{\bbb}{\mathfrak{b}}
\newcommand{\Diag}{\Delta\formal}
\newcommand{\Diagfinite}{\Delta^{\upscript{(r)}}}
\newcommand{\Xdiag}{X_{X \times X}\formal}
\newcommand{\Xdiagfinite}{X^{(r)}_{X \times X}}
\newcommand{\Shat}{\hat{S}}

\newcommand{\Sh}{\operatorname{Sh}}
\newcommand{\Aut}{\operatorname{Aut}}
\newcommand{\Spf}{\operatorname{Spf}}
\newcommand{\GL}{\mathbf{GL}}

\newcommand{\JetX}{\mathcal{J}^\infty_X}
\newcommand{\JetXfinite}{\mathcal{J}^r_X}
\newcommand{\graded}{\operatorname{gr}}
\newcommand{\Lie}{\mathcal{L}}
\newcommand{\EuclideanConn}{\nabla_\mathfrak{E}}

\newcommand{\upscript}[1]{{\scriptscriptstyle{#1}}}

\newcommand{\Yhat}{{\hat{Y}}}
\newcommand{\Yhatfinite}{{\hat{Y}^{\upscript{(r)}}}}

\newcommand{\Osheaf}{\mathscr{O}}

\newcommand{\Asheaf}{\mathscr{A}}
\newcommand{\Isheaf}{\mathscr{I}}

\newcommand{\formal}{^{\upscript{(\infty)}}}

\newcommand{\Linf}{L_\infty}

\newcommand{\proj}{\operatorname{pr}}

\newarrow{Equal} =====

\title[The Dolbeault dga of the formal neighborhood of the diagonal]
{The Dolbeault dga of the formal neighborhood of the diagonal}

\author{Shilin Yu}
\address{Department of Mathematics, University of Pennsylvania, PA 19104-6395, USA
}
\email{shilinyu@math.upenn.edu}


\thanks{This research was partially supported by the grant DMS1101382  from the National Science Foundation.}

\dedicatory{Dedicated to my parents}


\begin{document}

\begin{abstract}
  A well-known theorem of Kapranov states that the Atiyah class of the tangent bundle $TX$ of a complex manifold $X$ makes the shifted tangent bundle $TX[-1]$ into a Lie algebra object in the derived category $D(X)$. Moreover, he showed that there is an $\Linf$-algebra structure on the shifted Dolbeault resolution $(\A^{\bullet-1}_X(TX),\dbar)$ of $TX$ and wrote down the structure maps explicitly in the case when $X$ is K\"ahler. The corresponding Chevalley-Eilenberg complex is isomorphic to the Dolbeault resolution $(\AOD_X(\JetX),\dbar)$ of the jet bundle $\JetX$ via the construction of the holomorphic exponential map of the K\"ahler manifold.  In this paper, we show that $(\AOD_X(\JetX),\dbar)$ is naturally isomorphic to the Dolbeault dga $(\A^\bullet(\Xdiag),\dbar)$ associated to the formal neighborhood of the diagonal of $X \times X$ which we introduced in \cite{DolbeaultDGA}. We also give an alternative proof of Kapranov's theorem by obtaining an explicit formula for the pullback of functions via the holomorphic exponential map, which allows us to study the general case of an arbitrary embedding later.
\end{abstract}

\maketitle

\tableofcontents

\section{Introduction}

In this paper, we continue the study of the Dolbeault differential graded algebra (dga) of the formal neighborhood of a closed embedding of complex manifolds introduced in \cite{DolbeaultDGA}. One of initial motivations of the overall project is to give a concrete description of this Dolbeault dga in terms of the differential geometry of the embedding. As a starting point, the current paper provides a reformulation of the work of Kapranov \cite{Kapranov} on diagonal embeddings. Our new perspective will allow us to deal with the general case later elsewhere.

In \cite{DolbeaultDGA} we defined, for the formal neighborhood $\Yhat$ of a closed embedding of complex manifolds $i: X \to Y$, a Dolbeault dga $A=(\A^\bullet(\Yhat),\dbar)$. Following the work of Block \cite{Block1}, we assigned to $A$ a dg-category $\PA$ and proved that, in the case when $X$ is compact, $\PA$ serves as a dg-enhancement of the derived category $\Dcoh(\Yhat)$ of coherent sheaves over the formal neighborhood $\Yhat$. 

In order to construct explicitly certain objects we are interested in and compute the hom complexes between them in the dg-category $\PA$ (see the introductory part of \cite{DolbeaultDGA}), it is necessary to have a geometric description of the Dolbeault dga and, in particular, its differential. This is the problem that we shall study here and in subsequent papers.

In \cite{Kapranov}, Kapranov considered the case of a diagonal embedding $\Delta: X \hookrightarrow X \times X$. He argued that the holomorphic structure of the formal neighborhood $\Xdiag$ of the diagonal is encoded in the Dolbeault resolution $(\AOD(\JetX),\dbar)$ of the jet bundle $\JetX$. In \cite{DolbeaultDGA}, we generalized this observation to an arbitrary closed embedding $i: X \hookrightarrow Y$ and defined a general notion of the Dolbeault complex (or dga) $(\A^\bullet(\Yhat),\dbar)$ of the formal neighborhood $\hat{Y}$. We shall show in Section \ref{subsec:diag_jet} that our general definition indeed gives the Dolbeault resolution of the jet bundle when specialized to the case of a diagonal embedding.

Under the assumption that $X$ is K\"ahler, Kapranov used the Levi-Civita connection to construct a fiberwise holomorphic exponential map
\begin{equation}\label{eq:intro_exp}
  \expmap: X\formal_{TX} \to \Xdiag
\end{equation}
between $X\formal_{TX}$, the formal neighborhood of the zero section in the holomorphic tangent bundle $TX$, and $\Xdiag$, regarded as nonlinear fiber bundles over $X$. However, the exopnential map is not a biholomorphism and Kapranov described the pullback of holomorphic structure on $\Xdiag$ to $X\formal_{TX}$ via $\expmap$ using the Atiyah class of $TX$. Algebraically, the pullback map via $\expmap$ gives rise to an ismorphism of dgas
\begin{equation}\label{eq:intro_expmap_pullback}
  \expmap^*: (\AOD_X(\JetX),\dbar) \to (\AOD_X(\Shat(T^*X)), D),
\end{equation}
where $\Shat(T^*X)$ is the completed symmetric algebra of the cotangent bundle $T^*X$ and can be thought of as the pushforward of $\Osheaf_{X\formal_{TX}}$ via the natural projection $X\formal_{TX} \to X$. The differential $D$ encodes the pulled back holomorphic structure on $TX$ and it is not the same as the ordinary $\dbar$ on $\Shat(T^*X)$ induced from that of $T^*X$. Kapranov showed that one has to correct the usual $\dbar$ by the curvature and its higher covariant derivatives. Let $\nabla$ be the $(1,0)$-part of the Levi-Civita connection and define
\begin{displaymath}
  R_2 = R = [\dbar, \nabla] \in \A^1_X(\Hom(S^2 TX, TX))
\end{displaymath}
and
\begin{displaymath}
  R_n = \nabla^{n-2} R \in \A^{0,1}_X(\Hom(S^n TX, TX)), \quad n \geq 2,
\end{displaymath}
so that $R_2 = R$ is a Dolbeault representative of the Atiyah class $\alpha_{TX} \in \gext^1_X(S^2 TX, TX)$ of the tangent bundle (\cite{Atiyah}). Kapranov obtained an explicit formula for $D$:
\begin{equation}\label{eq:Kapranov_D}
  D = \dbar + \sum_{n \geq 2} \tilde{R}_n,
\end{equation}
where $\tilde{R}_n$ is the derivation on $\AOD(\Shat(T^*X))$ induced by $R_n$, regarded now as elements in $\A^{0,1}(\Hom(T^*X, S^n T^*X))$.

Kapranov's construction of the exponential map and proof of the formula \eqref{eq:Kapranov_D} for the differential $D$ is based on the formal geometry developed by Gelfand, Kazhdan and Fuchs (see \cite{GelfandKazhdan}, \cite{GelfandKazhdanFuks}). Nevertheless, we will give a very simple formula \eqref{eq:exp_Kahler} for the pullback map $\expmap^*$ in Section \ref{subsec:Kapranov} and we will use it to give an algebraic proof of Kapranov's theorem. This formula is handy for explict computation and later generalization to the case of general embeddings. 

As an interesting application of our formula for $\expmap^*$ (with slight adjustment), we will show in Section \ref{subsec:application} that the Atiyah class is the only obstruction for the existence of an isomorphism between the $\Osheaf_X$-sheaves $\JetX$ and $\Shat(T^*X)$ respecting the natural filtrations on both sides (Theorem \ref{thm:jet_Atiyah}), for an arbitrary complex manifold $X$ without the K\"ahler assumption. Note that it is not even obvious from Kapranov's formula \eqref{eq:Kapranov_D} of $D$ in the K\"ahler case, since the curvature $R$ does not necessarily vanishes on the nose even when its corresponding cohomology class $\alpha_{TX} \in \gext^1_X(S^2 TX, TX)$ is zero.

This little yet somewhat surprising fact has intimate connection with the work of Arinkin and C\u{a}ld\u{a}raru \cite{ArinkinCaldararu} and Calaque, C\u{a}ld\u{a}raru and Tu \cite{CalaqueCaldararuTu}, which is best understood in the language of derived algebraic geometry in the sense of Lurie \cite{Lurie}. We would also like to mention the recent work of \cite{ChenStienonXu} of Chen, Sti\'enon and Xu and the work of Calaque \cite{Calaque}. They considered the general case of a inclusion of Lie algebroids $A \subset L$, which is specialized to our case when $L = TX \otimes \complex$ is the complexified tangent bundle of a complex manifold $X$ and $A = T^{0,1}X$ is the $(0,1)$-part of $L$. In particular, Theorem \ref{thm:jet_Atiyah} can be regarded as the 'Koszul dual' of the Theorem 1.1 in \cite{Calaque}. However, we will not explore in details these relations in this paper.

To get ride of the K\"ahler assumption, Kapranov introduced the bundle $X_{exp}$ of formal exponential maps on $X$ (which he denoted as $\Phi(X)$). Any smooth section $\sigma$ of $X_{exp}$ gives a exponential map of the form \eqref{eq:intro_exp} and the corresponding pullback map \eqref{eq:intro_expmap_pullback}, where the differential $D$ can be characterized by a formula similar to \eqref{eq:Kapranov_D} containing analogues of $R_n$ which measure the failure of $\sigma$ to be holomorphic. We will review in Section \ref{sec:Formal} the construction of $X_{exp}$ and show that it can be naturally identified with another bundle $X_{conn}$ of jets of holomorphic connections which are flat and torsion-free, which is more related to our formula for the pullback map $\expmap^*$. The discussion will make it clear that our construction of the exponential map is equivalent to Kapranov's construction in the K\"ahler case.

The paper is organized as follows. In Section \ref{sec:Dolbeault}, we review the definitions and basic properties of the Dolbeault dga from \cite{DolbeaultDGA} and show that, in the case of a diagonal embedding, the Dolbeault dga is isomorphic to the Dolbeault resolution of the jet bundle considered in \cite{Kapranov}. In Section \ref{sec:Kapranov}, we state and prove our version of Kapranov's theorem. We will discuss the application mentioned above and make more comments about its relation with other works. Section \ref{sec:Formal} is devoted to setting up the language of formal geometry adopted in the paper. We mainly follow the nice approach of Bezrukavnikov and Kaledin \cite{BeKaledin}. At the end, we argue that our version of the holomorphic exponential map coincides that with Kapranov's original one.

\begin{thank}
  We would like to thank Jonathan Block, Damien Calaque, Andrei C\u{a}ld\u{a}raru, Nigel Higson, Mathieu Sti\'enon, Junwu Tu and Ping Xu for helpful discussions.
\end{thank}

\section{Dolbeault dga of a formal neighborhood}\label{sec:Dolbeault}

\subsection{Definitions and notations}\label{subsec:defn}

Let $Y$ be a complex manifold and $X$ be a closed complex submanifold in $Y$. We denote by $i: X \hookrightarrow Y$ the embedding, by $\Osheaf_X$ and $\Osheaf_Y$ the structure sheaf of germs of holomorphic functions over $X$ and $Y$ respectively, by $\Isheaf$ the ideal sheaf of $\Osheaf_Y$ of holomorphic functions vanishing along $X$. Then we can define \emph{the $r$-th formal neighborhood $\Yhatfinite$} of the embedding as the ringed space $(X, \Osheaf_{\Yhatfinite})$ of which the structure sheaf is
\begin{displaymath}
  \Osheaf_{\Yhatfinite} = \Osheaf_Y / \Isheaf^{r+1}.
\end{displaymath}
\emph{The (complete) formal neighborhood $\Yhat = \Yhat^{\upscript{(\infty)}}$} is defined to be the ringed space $(X, \Osheaf_{\Yhat})$ where
\begin{displaymath}
 \Osheaf_{\Yhat} = \varprojlim_{r} \Osheaf_{\Yhatfinite} = \varprojlim_{r} \Osheaf_X / \Isheaf^{r+1}.
\end{displaymath}
We will also write $X^{\upscript{(\infty)}}_Y$ and $X^{\upscript{(r)}}_Y$ instead of $\Yhat$ and $\Yhatfinite$ to emphasize the submanifolds in question.

We review the notion of \emph{Dolbeault differential graded algebra (dga)} of the embedding $i: X \hookrightarrow Y$ from \cite{DolbeaultDGA}. Let $(\AOD(Y), \dbar) = (\wedge^\bullet \Omega_Y^{0,1}, \dbar)$ denote the Dolbeault dga of $Y$, with multiplication being the wedge product and the differential
\begin{displaymath}
  \dbar: \AOD(Y) \to \A^{0,\bullet+1}(Y)
\end{displaymath}
being the $(0,1)$-part of the de Rham differential. For each nonnegative integer $r$, we set $\aaa^\bullet_r$ to be the graded ideal of $\AOD(Y)$ consisting of those forms $\omega \in \AOD(Y)$ satisfying
\begin{equation}\label{defn:aaa_finite}
  i^*(\Lie_{V_1} \Lie_{V_2} \cdots \Lie_{V_l} \omega) = 0, \quad \forall~ 1 \leq j \leq l,
\end{equation}
for any collection of smooth $(1,0)$-vector fields $V_1, V_2, \ldots, V_l$ over $Y$, where $0 \leq l \leq r$.

\begin{proposition}[\cite{DolbeaultDGA}]
  The ideal $\aaa^\bullet_r$ is invariant under the action of $\dbar$, hence is a dg-ideal of $(\AOD(Y),\dbar)$.
\end{proposition}

\begin{definition}
  The \emph{Dolbeault dga of $\Yhatfinite$} is the quotient dga
  \begin{displaymath}
    \A^\bullet(\Yhatfinite) := \AOD(Y) / \aaa^\bullet_r.
  \end{displaymath}
  In particular, $\A^\bullet(\Yhat^{\upscript{(0)}}) = \AOD(X)$. Moreover, there is a descending filtration of dg-ideals
  \begin{displaymath}
    \aaa^\bullet_0 \supset \aaa^\bullet_1 \supset \aaa^\bullet_2 \supset \cdots,
  \end{displaymath}
  which induces an inverse system of dgas with surjective connecting morphisms
  \begin{displaymath}
    \AOD(X) = \A^\bullet(\Yhat^{\upscript{(0)}}) \leftarrow \A^\bullet(\Yhat^{\upscript{(1)}}) \leftarrow \A^\bullet(\Yhat^{\upscript{(2)}}) \leftarrow \cdots.
  \end{displaymath}
  The \emph{Dolbeault dga of $\Yhat$} is defined to be the inverse limit
  \begin{displaymath}
    \A^\bullet(\Yhat) = \A^\bullet(\Yhat^{\upscript{(\infty)}}):= \varprojlim_r \A^\bullet(\Yhatfinite).
  \end{displaymath}
  We will write $\A(\Yhat) = \A^0(\Yhat)$ and $\A(\Yhatfinite) = \A^0(\Yhatfinite)$ for the zeroth components of the Dolbeault dgas.
\end{definition}

By Remark 2.4., \cite{DolbeaultDGA}, we have the following alternative description of $(\A^\bullet(\Yhat),\dbar)$.

\begin{proposition}
  The natural map
  \begin{equation}\label{eq:Ahat_quotient}
    \AOD(Y) \Big/ \bigcap_{r \in \NN} \aaa^\bullet_r \to \A^\bullet(\Yhat)
  \end{equation}
  induced by the quotient maps
  \begin{displaymath}
    \AOD(Y) \to \A^\bullet(\Yhatfinite), \quad r \in \NN,
  \end{displaymath}
  is an isomorphism of dgas.
\end{proposition}

In \cite{DolbeaultDGA}, we sheafified the Dolbeault dgas $\A^\bullet(\Yhatfinite)$ and $\A^\bullet(\Yhat)$ to obtain sheaves of dgas $\Asheaf^\bullet(\Yhatfinite)$ and $\Asheaf^\bullet(\Yhat)$ respectively over $X$. Similarly, $\aaa^\bullet_r$ induces a sheaf of dg-ideals $\widetilde{\aaa}^\bullet_r$. Moreover, there are natural inclusions of sheaves of algebras $\Osheaf_{\Yhatfinite} \hookrightarrow \Asheaf(\Yhatfinite)$ and $\Osheaf_{\Yhat} \hookrightarrow \Asheaf(\Yhat)$. The following result was proved in \cite{DolbeaultDGA}.

\begin{theorem}
  For any nonnegative integer $r$ or $r = \infty$, the complex of sheaves
  \begin{displaymath}
    0 \to \Osheaf_{\Yhatfinite} \to \Asheaf^0_{\Yhatfinite} \xrightarrow{\dbar} \Asheaf^1_{\Yhatfinite} \xrightarrow{\dbar} \cdots \xrightarrow{\dbar} \Asheaf^m_{\Yhatfinite} \to 0
  \end{displaymath}
  is exact, where $m = \dim X$. In other words, $(\Asheaf^\bullet_{\Yhatfinite},\dbar)$ is a fine resolution of $\Osheaf_{\Yhatfinite}$.
\end{theorem}

As the completion of $\AOD(Y)$ with respect to the filtration $\aaa^\bullet_r$, the dga $\A^\bullet(\Yhat)$ is itself filtered and its associated graded dga is
\begin{displaymath}
  \graded \A^\bullet(\Yhat) \simeq (\AOD(Y) / \aaa^\bullet_0) \oplus \bigoplus_{r=0}^{\infty} \aaa^\bullet_r / \aaa^\bullet_{r+1}.
\end{displaymath}
Note that $\AOD(X) \simeq \AOD(Y) / \aaa^\bullet_0$ and $\aaa^\bullet_r / \aaa^\bullet_{r+1}$ are dg-modules over $(\AOD(X),\dbar)$. We define, for each $r \geq 0$, a `cosymbol map' of complexes
\begin{equation}\label{eq:isom_assoc_graded}
  \tau_r : (\aaa^\bullet_r / \aaa^\bullet_{r+1}) \to \AOD_X(S^{r+1} \conormal),
\end{equation}
where $S^{r+1} \conormal$ is the $(r+1)$-fold symmetric tensor of the conormal bundle $\conormal$ of the embedding. Given any $(r+1)$-tuple of smooth sections $\mu_1, \ldots, \mu_{r+1}$ of $\normal$, we lift them to smooth sections of $TY|_X$ and extend to smooth $(1,0)$-tangent vector fields $\tilde{\mu}_1, \ldots, \tilde{\mu}_{r+1}$ on $Y$ (defined near $X$). We then define the image of $\omega + \aaa^\bullet_{r+1} \in \aaa^\bullet_r / \aaa^\bullet_{r+1}$ under $\tau$ for any $\omega \in \aaa^\bullet_{r}$, thought of as linear functionals on $(\normal)^{\otimes (r+1)}$, by the formula
\begin{equation}
  \tau_r (\omega + \aaa^\bullet_{r+1}) (\mu_1 \otimes \cdots \otimes \mu_{r+1}) = i^* \Lie_{\tilde{\mu}_1} \Lie_{\tilde{\mu}_2} \cdots \Lie_{\tilde{\mu}_{r+1}} \omega.
\end{equation}
The map $\tau_r$ is well-defined and is independent of the choice of the representative $\omega$ and $\tilde{\mu}_j$'s. Moreover, the tensor part of $\tau_r (\omega + \aaa^\bullet_{r+1})$ is indeed symmetric.

\begin{proposition}
  The map $\tau_r$ in \eqref{eq:isom_assoc_graded} is an isomorphism of dg-modules over $(\AOD(X),\dbar)$.
\end{proposition}

\begin{corollary}
  We have a natural isomorphism of dgas
  \begin{displaymath}
    \graded \A^\bullet(\Yhat) \simeq \bigoplus_{n=0}^{\infty} \AOD_X(S^n \conormal).
  \end{displaymath}
\end{corollary}

\begin{remark}
 Since $\Osheaf_{\Yhat}$ is defined as an inverse limit sheaf, it has a natural descending filtration and so $\Yhat$ should be regarded not only as a ringed space but also as a topologically ringed space. Similarly, defined as an inverse limit, the Dolbeault dga $\A^\bullet(\Yhat)$ is a filtered dga. Moreover, since $\AOD(Y)$ is a Fr\'echet dga and all ideals $\aaa^\bullet_r$ are closed, the inverse limit $\A^\bullet(\Yhat)$ can be made into a Fr\'echet dga equipped with the initial topology (which concides with the quotient Fr\'echet topology via the quotient map \eqref{eq:Ahat_quotient}). One can recover the formal neighborhood $\Yhat$ from the topological dga $\A^\bullet(\Yhat)$. The topology or the filtration would matter when one wants to consider morphisms between two such dgas or dg-modules over the Dolbeault dgas (see \cite{DolbeaultDGA}). However, these are not among the topics of the current paper. The reader only needs to keep in mind that  in this paper every morphism between two filtered dgas will preserve the filtrations. 
\end{remark}

\subsection{Diagonal embeddings and jet bundles}\label{subsec:diag_jet}

We consider the case of diagonal embeddings. Let $X$ be a complex manifold and let $\Delta: X \hookrightarrow X \times X$ be the diagonal map. For convenience, the associated formal neighborhoods are denoted as $\Diag = \Xdiag$ and $\Diagfinite = \Xdiagfinite$ throughout this section. We denote by $\proj_1, \proj_2 : X \times X \to X$ the projections onto the first and second component of $X \times X$ respectively. From the algebraic perspective, the jet bundle $\JetXfinite$ of order $r$ ($r \geq 0$) can be viewed as the sheaf of algebras
\begin{displaymath}
  \JetXfinite = \proj_{1*} \Osheaf_{\Diagfinite},
\end{displaymath}
and similarly for the jet bundle $\JetX$ of infinite order,
\begin{displaymath}
  \JetX = \proj_{1*} \Osheaf_{\Diag}.
\end{displaymath}
Moreover, they are sheaves of $\Osheaf_X$-modules where the $\Osheaf_X$-actions are induced from the projection $\proj_1$. Analytically, the $\Osheaf_X$-modules $\JetXfinite$ can be regarded as finite dimensional holomorphic vector bundles, while $\JetX$ is a holomorphic vector bundle of infinite dimension and is the projective limit of $\JetXfinite$'s. Fiber of $\JetXfinite$ (resp. $\JetX$) at each point $x \in X$ can be naturally identified with the algebra of holomorphic $r$-jets (resp. $\infty$-jets) of functions at $x$. Thus we can form the sheaf of Dolbeault complexes of $\JetXfinite$,
\begin{displaymath}
  (\Asheaf^{0,\bullet}(\JetXfinite), \dbar) = (\JetXfinite \otimes_{\Osheaf_X} \Asheaf^{0,\bullet}_X, 1 \otimes \dbar),
\end{displaymath}
where $(\Asheaf^{0,\bullet}_X, \dbar)$ is the sheaf of Dolbeault complexes over $X$. Moreover, $(\Asheaf^{0,\bullet}(\JetXfinite), \dbar)$ is also a sheaf of dgas, such that the multiplication is induced by that of $\JetXfinite$ and the wedge product of forms. The sheaf of Dolbeault complexes of $\JetX$ is defined to be
\begin{displaymath}
  \Asheaf^{0,\bullet}(\JetX) = \varprojlim_{r} \Asheaf^{0,\bullet}(\JetXfinite) = \varprojlim_{r} \Asheaf^{0,\bullet}_X \otimes_{\Osheaf_X} \JetXfinite,
\end{displaymath}
which is also a sheaf of dgas. We also denote the global sections of the sheaves by
\begin{displaymath}
  \AOD(\JetXfinite) = \Gamma(X, \Asheaf^{0,\bullet}(\JetXfinite)), \quad \AOD(\JetX) = \Gamma(X, \Asheaf^{0,\bullet}(\JetX)).
\end{displaymath}

The Dolbeault dga $(\A^\bullet(\Diagfinite),\dbar)$ is also an $(\AOD(X),\dbar)$-dga. The $\AOD(X)$-action is given by the compositions of homomorphisms of dgas
\begin{displaymath}
  \AOD(X) \xrightarrow{\proj_1^*} \AOD(X \times X) \to \A^\bullet(\Diagfinite).
\end{displaymath}
Similarly, we have natural homomorphisms of dgas
\begin{displaymath}
  (\AOD(X),\dbar) \to (\A^\bullet(\Diag),\dbar), \quad (\Asheaf^{0,\bullet}_X,\dbar) \to (\Asheaf^\bullet_{\Diag},\dbar), \quad (\Asheaf^{0,\bullet}_X,\dbar) \to (\Asheaf^\bullet_{\Diagfinite}, \dbar),
\end{displaymath}
which are all induced by $\proj_1$. The natural inclusion $\Osheaf_{\Diagfinite} \hookrightarrow \Asheaf_{\Diagfinite}$ can be regarded as the $\Osheaf_X$-linear morphism
\begin{displaymath}
  \JetXfinite \hookrightarrow \Asheaf_{\Diagfinite},
\end{displaymath}
which extends to an $(\Asheaf^{0,\bullet}_X,\dbar)$-linear morphism of sheaves of dgas
\begin{equation}\label{eq:isom_jet_finite}
  I_r: \Asheaf^{0,\bullet}(\JetXfinite) = \Asheaf^{0,\bullet}_X \otimes_{\Osheaf_X} \JetXfinite \to \Asheaf^\bullet_{\Diagfinite}.
\end{equation}
By letting $r$ vary and passing to the inverse limit, we obtain an $(\Asheaf^{0,\bullet}_X,\dbar)$-morphism of sheaves of dgas
\begin{equation}\label{eq:isom_jet_infinite}
  I_{\infty}: \Asheaf^{0,\bullet}(\JetX) \to \Asheaf^\bullet_{\Diag}.
\end{equation}
We use the same notations for maps induced by $I_r$ and $I_{\infty}$ on the global sections.

\begin{proposition}\label{prop:isom_diag_dga}
  The homomorphisms
  \begin{displaymath}
    I_r: \AOD(\JetXfinite) \xrightarrow{\simeq} \A^\bullet(\Diagfinite)
  \end{displaymath}
  and
  \begin{displaymath}
    I_{\infty}: \AOD(\JetX) \xrightarrow{\simeq} \A^\bullet(\Diag).
  \end{displaymath}
  are isomorphisms of $(\AOD(X),\dbar)$-dgas. Similar results hold for the corresponding sheaves.
\end{proposition}

\begin{proof}
  We prove the proposition by induction on $r \geq 0$. When $r=0$, both $\AOD(\mathcal{J}^{0}_X)$ and $\A^\bullet(\Delta^{\upscript{(0)}})$ can be identified with $\AOD(X)$ and $I_0$ is an isomorphism. Now assume that $I_r$ is an isomorphism. We have for any $k \geq 0$ a natural isomorphism $\Isheaf^k / \Isheaf^{k+1} \simeq S^k T^*X$ of sheaves of $\Osheaf_X$-modules similar to $\tau_r$ in \eqref{eq:isom_assoc_graded} and hence a short exact sequence
  \begin{displaymath}
    0 \to S^k T^*X \to \mathcal{J}^k_X \to \mathcal{J}^{k-1}_X \to 0.
  \end{displaymath}
  Now there is a commutative diagram
  \begin{diagram}
    & 0 & \rTo & \AOD(S^{r+1}T^*X)   & \rTo & \AOD(\mathcal{J}^{r+1}_X) & \rTo & \AOD(\mathcal{J}^r_X) & \rTo & 0  \\
    &   &      & \dTo_{\tau_r^{-1}}  &      & \dTo_{I_{r+1}}            &      & \dTo_{I_r}   \\
    & 0 & \rTo & \aaa^\bullet_{r} / \aaa^\bullet_{r+1} & \rTo & \A^\bullet(\Delta^{\upscript{(r+1)}}) & \rTo & \A^\bullet(\Delta^{\upscript{(r)}}) & \rTo & 0
  \end{diagram}
  of two rows of short exact sequences. The leftmost vertical map is exactly the inverse of $\tau_r$ in \eqref{eq:isom_assoc_graded}, where the conormal bundle of the diagonal is identified with the cotangent bundle of $X$. The rightmost vertical map is $I_r$, which is an isomorphism by the inductive hypothesis. Therefore, by the five lemma, $I_{r+1}$ in the middle is also an isomorphism.
\end{proof}

\section{Kapranov's theorem revisited}\label{sec:Kapranov}

\subsection{Conventions on the symmetric algebra}\label{subsec:symmetric}

We recall a basic yet useful observation from \cite{ArinkinCaldararu}. Let $V$ be a finite dimensional vector space. The free coalgebra $T^{c}(V) = \oplus_{k \geq 0} V^{\otimes k}$ generated by $V$ is equipped with a commutative algebra structure, give by the shuffle product
\[ (v_1 \otimes \cdots \otimes v_p) \cdot (v_{p+1} \otimes \cdots \otimes v_{p+1}) = \sum_{\sigma \in \Sh(p,q)} v_{\sigma(1)} \otimes \cdots \otimes v_{\sigma(p+q)}, \]
where $\Sh(p,q)$ is the set of all $(p,q)$-shuffles.

The symmetric algebra $S(V) = \oplus_{k \geq 0} S^k V$ is naturally a subalgebra and quotient algebra of $T^c(V)$ via the inclusion
\begin{equation}\label{eq:sym_inclusion}
  S(V) \to T^c(V), \quad v_1 v_2 \cdots v_k \mapsto \sum_{\sigma \in \Sigma_k} v_{\sigma(1)} \otimes v_{\sigma(2)} \otimes \cdots \otimes v_{\sigma(k)} 
\end{equation}
and the surjection
\begin{equation}\label{eq:sym_quotient}
  T^c(V) \to S(V), \quad v_1 \otimes v_2 \otimes \cdots \otimes v_k \mapsto \frac{1}{k!} v_1 v_2 \cdots v_k.
\end{equation}
Both are homomorphism of commutative algebras and moreover the composition
\[ S(V) \to T^c(V) \to S(V) \]
is the identity. The same discussion applies to the completions $\hat{T}^c(V)$ and $\hat{S}(V)$ with respect to the natural gradings. From now on, we will always think of $\hat{S}(V)$ as a subalgebra of $T^c(V)$ via the inclusion defined above.

\subsection{Kapranov's theorem and holomorphic exponential map}\label{subsec:Kapranov}

We now reformulate and provide an alternative proof of Kapranov's theorem on the concrete description of the formal neighborhood of a diagonal embedding in terms of the Atiyah class. Suppose that $X$ is equipped with a K\"ahler metric $h$. Let $\nabla$ be the canonical $(1,0)$-connection in $TX$ associated with $h$, so that
\begin{equation}\label{eq:flatness}
  [\nabla,\nabla] = 0 ~\text{in}~ \A^{2,0}_X(\End(TX)).
\end{equation}
and it is torsion-free, which is equivalent to the condition for $h$ to be K\"ahler.

Set $\widetilde{\nabla} = \nabla + \dbar$, where $\dbar$ is the $(0,1)$-connection defining the complex structure. The curvature of $\widetilde{\nabla}$ is
\begin{equation}\label{defn:AtiyahClass}
  R = [\dbar, \nabla] \in \A^{1,1}_X(\End(TX)) = \A^{0,1}_X(\Hom(TX \otimes TX, TX)).
\end{equation}
In fact, by the torsion-freeness we have
\begin{displaymath}
  R \in \A^{0,1}_X(\Hom(S^2TX,TX)).
\end{displaymath}
We have the Bianchi identity
\begin{displaymath}
  \dbar R = 0 ~ \text{in} ~ \A^{0,2}_X(\Hom(S^2 TX, TX)).
\end{displaymath}
Thus $R$ defines a cohomology class in $\gext^1_{\Osheaf_X}(S^2 TX, TX)$, which is the \emph{Atiyah class} $\alpha_{TX}$ of the tangent bundle (\cite{Atiyah}). It is the obstruction for the existence of a holomorphic $(1,0)$-connection on $TX$.

Now define tensor fields $R_n$, $n \geq 2$, as higher covariant derivatives of the curvature:
\begin{equation}\label{defn:Derivative_AtiyahClass}
  R_n \in \A^{0,1}_X(\Hom(S^2TX \otimes TX^{\otimes(n-2)}, TX)), \quad R_2:=R, \quad R_{i+1}=\nabla R_i.
\end{equation}
In fact $R_n$ is totally symmetric, i.e.,
\begin{displaymath}
  R_n \in \A^{0,1}_X(\Hom(S^n TX,TX)) = \A^{0,1}_X(\Hom(T^*X, S^n T^*X))
\end{displaymath}
by the flatness of $\nabla$ \eqref{eq:flatness}. Note that if we think of $\nabla$ as the induced connection on the cotangent bundle, the same formulas \eqref{defn:AtiyahClass} and \eqref{defn:Derivative_AtiyahClass} give $-R_n$.

Kapranov observed that one can define a differential
  \begin{equation}\label{eq:Kapranov_differential}
    D = \dbar + \sum_{n \geq 2} \tilde{R}_n
  \end{equation}
on the graded algebra $\AOD_X(\Shat(T^*X))$, where $\tilde{R_n}$ is the odd derivation of $\AOD_X(\Shat(T^*X))$ induced by $R_n \in \A^{0,1}_X(\Hom(T^*X, S^n T^*X))$. (Note that in \cite{Kapranov} Kapranov used the notation $\tilde{R}^*_n$.) By purely algebraic properties of $R_n$'s, it was shown that $D^2 = 0$. In other words, $(\AOD_X(\Shat(T^*X), D)$ is a dga. Kapranov gave a geometric intepretation of this dga (following the suggestion by Ginzburg). Namely, he constructed a `holomorphic exponential map'
  \[ \expmap: X\formal_{TX} \to \Xdiag \]
using the Levi-Civita connection of $X$, where $X\formal_{TX}$ is the formal neighborhood of $X$ (regarded as the zero section) in the total space of $TX$. The exponential $\expmap$ is a fiberwise holomorphic morphism between (nonlinear) holomorphic fiber bundles, but it is not holomorphic when the base point on $X$ moves around. If we think of $\Shat(T^*X)$ as sheaves of functions on $X\formal_{TX}$, then $D^2 = 0$ means that $D$ defines a new holomorphic structure on $X\formal_{TX}$, which is exactly the holomorphic structure on $\Xdiag$ pushed forward via $\expmap$. In other words, the pullback map via $\expmap$ induces an isomorphism of dgas
\begin{equation}\label{eq:Kapranov_isom}
  \expmap^*: (\AOD_X(\JetX),\dbar) \to (\AOD_X(\Shat(T^*X)), D).
\end{equation}

As mentioned in the Introduction, we will provide an alternative proof of Kapranov's theorem by writing down directly a formula for the pullback map $\expmap^*$. The reason why it coincides with Kapranov's original construction of the exponential map will be clear in Section \ref{sec:Formal}. 

Since we have shown in Section \ref{subsec:diag_jet} that our Dolbeault dga $(\A^\bullet(\Xdiag), \dbar)$ is isomorphic to $(\AOD_X(\JetX), \dbar)$, we can work with the former. Define
\begin{displaymath}
  \expmap^* : \A^\bullet(\Xdiag) \xrightarrow{\simeq} \AOD_X(\Shat(T^*X))
\end{displaymath}
by
\begin{equation}\label{eq:exp_Kahler}
  \expmap^* ([\eta]_\infty) = (\Delta^*\eta, \Delta^*\nabla \eta, \Delta^*\nabla^2 \eta, \cdots, \Delta^*\nabla^n \eta, \cdots),
\end{equation}
where $\eta \in \AOD(X \times X)$ and $[\eta]_\infty$ is its image in $\A^\bullet(\Xdiag)$. Here $\nabla$ is understood as the pullback of the original $\nabla$ of $T^*X$ to $X \times X$ via $\proj_2: X \times X \to X$. In other words, it is now a constant family of connections on the trivial fiber bundle $\proj_1 : X \times X \to X$ which act only in the direction of the fibers (i.e., the second factor of $X \times X$). By $\nabla \eta$ we mean the $(1,0)$-differential $\partial_2 \eta$ of $\eta$ in the direction of the second factor of $X \times X$ and $\nabla^n \eta = \nabla^{n-1} \partial_2 \eta$. 

The map $\expmap^*$ is an isomorphism of filtrated graded algebras since it induces the identity map on the associated graded algebras (both equal to $\AOD_X(S(T^*X))$).

\begin{remark}
  The reader might think that we should add the factor $1/k!$ in each component of the formula \eqref{eq:exp_Kahler} so that it will look the same as the usual Taylor expansion. Yet what is really going on here is that \emph{a priori} $(\Delta^* \nabla^k \eta)_{k \geq 0}$ belongs to $\AOD_X(\hat{T}^c(T^*X))$, then by the torsion-freeness and flatness of $\nabla$ we know it actually lies in the symmetic part of $\AOD_X(\hat{T}^c(T^*X))$. To extract it as an element in $\AOD_X(\hat{S}(T^*X))$ we apply the quotient map \eqref{eq:sym_quotient} in Section \ref{subsec:symmetric} to $\hat{T}^c(T^*X)$ part and the factors $1/k!$ arise exactly from there. In particular, when $X = \complex^n$ with the flat K\"ahler metric, we get the usual Taylor expansions. We will not repeat this point when similar situations appear later.
\end{remark}

\begin{theorem}[compare with Theorem 2.8.2, \cite{Kapranov}]\label{thm:Kapranov_Diagonal}
  Assume that $X$ is K\"ahler. The map
  \begin{displaymath}
    \expmap^* : (\A^\bullet(\Xdiag),\dbar) \to (\AOD_X(\Shat(T^*X)), D)
  \end{displaymath}
  is an isomorphism of dgas, where $D$ is as defined in \eqref{eq:Kapranov_differential}. In particular, $D^2 = 0$. 
\end{theorem}

\begin{proof}
 We need to show that
   \begin{equation}\label{eq:exp_commutator}
     \dbar \circ \expmap^* - \expmap^* \circ ~\dbar  = - (\sum_{n \geq 2} \tilde{R}_n) \circ \expmap^*. 
   \end{equation}
 Let
  \begin{displaymath}
    \Upsilon = [\dbar,\nabla] \in \A^{0,1}_{X \times X} (\Hom(\proj_2^* T^*X, S^n(\proj_2^* T^*X)))
  \end{displaymath}
  and
  \begin{displaymath}
    \widetilde{\Upsilon} \in \A^{0,1}_{X \times X}(\Hom(\Shat^\bullet(\proj_2^* T^*X), \Shat^{\bullet+1}(\proj_2^* T^*X)))
  \end{displaymath}
  the derivation on $\A^{0,\bullet}_{X \times X} (\Shat^{\bullet}(\proj_2^* T^*X))$ induced by $\Upsilon$ which increases the degree on $\Shat^\bullet$ by $1$. Here again by $\nabla$ we mean the pullback $\proj_2^* \nabla$ acting on the cotangent bundle. But it is constant in the direction of the first factor of $X \times X$. So if we decompose $\dbar = \dbar_1 + \dbar_2$ according to the product $X \times X$, we get
  \begin{displaymath}
    \Upsilon = [\dbar,\nabla] = [\dbar_2, \nabla] = \proj_2^*(-R).
  \end{displaymath}
  For any $[\eta]_\infty \in \A^\bullet(\Xdiag)$, we have
  \begin{displaymath}
    \nabla^n \dbar \eta = \dbar \nabla^n \eta - \sum_{i+j=n-2} \nabla^i \circ \widetilde{\Upsilon} \circ \nabla^j \eta.
  \end{displaymath}
  By evaluating $\widetilde{\Upsilon}$ and expanding the action of $\nabla^i$ via the Leibniz rule, we find
  \begin{displaymath}
    \nabla^n \dbar \eta = \dbar \nabla^n \eta - \sum_{k=0}^{n-1} \widetilde{\nabla^{k} \Upsilon} \circ \nabla^{n-k-1} \eta,
  \end{displaymath}
  where $\widetilde{\nabla^k \Upsilon}$ is the derivation induced by $\nabla^k \Upsilon$. Finally by applying $\Delta^*$ on both sides we obtain \eqref{eq:exp_commutator}, since $\Delta^* \nabla^k \Upsilon = \Delta^* \proj_2^* (-R_k) = -R_k$. 
\end{proof}

\subsection{An application}\label{subsec:application}

As an application of the formula \eqref{eq:exp_Kahler} (yet in a more general context), we prove the following result claiming essentially that the Atiyah class is the only obstruction for the existence of a biholomorphism between $\Xdiag$ and $X\formal_{TX}$, which might look surprising at first sight.

\begin{theorem}\label{thm:jet_Atiyah}
  Let $X$ be an arbitrary complex manifold (not necessarily K\"ahler). Then there exists an $(\AOD(X),\dbar)$-linear isomorphism of dgas
  \[ (\A^\bullet(\Xdiag),\dbar) \xrightarrow{\simeq} (\AOD_X(\Shat(T^*X)), \dbar) \] 
preserving the natural filtrations on both sides and inducing the identity map on the associated graded algebras (both equal to $\AOD_X(S(T^*X))$), if and only if the Atiyah class $\alpha_{TX}$ of $TX$ vanishes.
\end{theorem}

\begin{proof}
  Assume such an isomorphism exists. In particular, it gives a holomorphic splitting of the short exact sequence
  \[ 0 \to S^2 T^*X \to \mathcal{J}^2_X  \to \mathcal{J}^1_X  \to 0 \]
and hence a holomorphic splitting for
  \[ 0 \to S^2 T^*X \to \Isheaf / \Isheaf^3 \to \Isheaf / \Isheaf^2 = T^*X \to 0, \]
where $\Isheaf \subset \Osheaf_{X \times X}$ is the ideal sheaf of functions vanishing along the diagonal. Therefore the Atiyah class vanishes by Proposition 2.2.1., \cite{Kapranov}.

  On the other hand, if the Atiyah class of $TX$, or equivalently, of $T^*X$ vanishes, then there exists a holomorphic connection $\nabla$ on $T^*X$ (cf. \cite{Atiyah}). We define a map from $\A^\bullet(\Xdiag)$ to the Dolbeault resolution of the sheaf of completed free coalgebras $\hat{T}^c(T^*X) \allowbreak = \prod_{k \geq 0} (T^*X)^{\otimes k}$,
  \[ I: \A^\bullet(\Xdiag) \to \AOD_X(\hat{T}^c(T^*X)) \]
by the formula
  \[ I([\eta]_{\infty}) = (\Delta^* \nabla^k \eta)_{k \geq 0} \in \AOD_X(\hat{T}^c(T^*X)) \]
which is similar to \eqref{eq:exp_Kahler} and $\nabla$ here is again understood as acting on the second factor of $X \times X$, etc. $I$ is a morphism of graded commutative algebras if we equip $\hat{T}^c(T^*X)$ with the shuffle product. Moreover, it commutes with the $\dbar$-derivation on both sides since $[\dbar, \nabla] = 0$. Thus $I$ is a morphism of dgas. 

The image of $I$ is no longer guaranteed to lie in the symmetric algebra since we do not have additional assumption on $\nabla$ as before. Yet we can compose $I$ with the quotient map \eqref{eq:sym_quotient} to get another morphism of graded algebras
\[ \tilde{I} : \A^\bullet(\Xdiag) \xrightarrow{I} \AOD_X(\hat{T}^c(T^*X)) \to \AOD_X(\Shat(T^*X)). \]
Since the quotient map $T^c(T^*X) \to \Shat(T^*X)$ is essentially the symmetrization map, it commutes with the $\dbar$-derivations on both sides induced by that of $T^*X$ and hence $\tilde{I}$ is a morphism of dgas. It also preserves the filtrations and its induced map on the associated graded algebras is the identity map, since symmetrization does not affect the symbols of the differential operators $\nabla^k$. Therefore $\tilde{I}$ fullfills the required conditions.
\end{proof}

\begin{digression}
  The result of Theorem \ref{thm:jet_Atiyah} is parallel to the work of Arinkin and C{\u{a}}ld{\u{a}}raru \cite{ArinkinCaldararu}. They showed there that, for a closed embedding $i: X \hookrightarrow Y$ of schemes, there exists an isomorphism
  \[i^*i_* \Osheaf_X  \simeq \bigoplus_k \wedge^k \conormal[k] \] 
of algebra objects in the derived category $D(X)$ of $X$, where $\conormal$ is the conormal bundle, if and only if a certain class $\alpha_N \in \gext^2_X(\wedge^2 \normal, \normal) = \gext^1_X(S^2(\normal[-1]), \normal[-1])$ (which they called the HKR class) of the embedding vanishes. $\alpha_N$ can be thought of as a relative version of Atiyah class of the tangent complex of the embedding, which is $\normal[-1]$. Moreover, they showed that $\alpha_N$ is the obstruction class for the existence of a holomorphic extension of the normal bundle $\normal$ to the first-order formal neighborhood $X^{\upscript{(1)}}_Y$. From the viewpoint of derived algebraic geometry (in the sense of Lurie \cite{Lurie}), $i^*i_* \Osheaf_X$ can be regarded as the structure complex of the derived self-intersection $X \times^R_Y X$.

In the context of Theorem \ref{thm:jet_Atiyah}, the corresponding morphism of spaces behind the scene is $X \to X_{dR}$. $X_{dR}=(X, \Omega^\bullet_X)$ is the de Rham space of $X$, which is $X$ equipped with the sheaf of de Rham complex $\Omega^\bullet_X$, regarded as a dg-manifold. The canonical map $X \to X_{dR}$ arises from the quotient map $\Omega^\bullet_X \to \Osheaf_X$. The derived fiber product $X \times^R_{X_{dR}} X$ is exactly $\Xdiag$. The candidate for the 'first-order formal neighborhood' in this case seems to be $X'_{dR} = (X, \Omega^{\leq 1}_X)$, where $\Omega^{\leq 1}_X$ is the truncated de Rham complex $\Osheaf_X \xrightarrow{d} \Omega^1_X$. The tangent complex of the map $X \to X_{dR}$ is nothing but $TX$ and the existence of an extension of the vector bundle $TX$ to $X'_{dR}$ is clearly equivalent to the existence of a holomorphic connection of $TX$ and hence equivalent to the vanishing to the usual Atiyah class of $TX$.  

The comparison between these two situations suggests that there should be an alternative proof of Theorem \ref{thm:jet_Atiyah} which is parallel to that in \cite{ArinkinCaldararu}. This was done in the paper of Calaque \cite{Calaque} in a more general context of an inclusion $A \subset L$ of sheaves of Lie algebroids over a space $X$ with a sheaf of algebras $R$, first considered by Chen, Sti\'enon and Xu \cite{ChenStienonXu}. The Chen-Sti\'enon-Xu class $\alpha_{L/A} \in \gext^1_A ((L/A) \otimes_R (L/A), L/A)$ defined in \cite{ChenStienonXu} recovers in particular the Atiyah class when $L = TX_{\real} \otimes \complex$ is (the sheaf of smooth sections of) the complexified tangent bundle of a complex manifold $X$ and $A=T^{0,1}X$. Calaque then defined for such an inclusion the first formal neighborhood $A^{(1)}$ as certain quotient of the free algebroid $FR(L)$ generated by $L$ (after Kapranov \cite{FreeAlgebroid}), and showed that the following statements are equivalent:
\begin{enumerate}
  \item
    The Chen-Stienon-Xu class $\alpha_{L/A}$ for $(L,A)$ vanishes.
  \item
    The $A$-module structure $L/A$ lifts to an $A^{(1)}$-module structure.
  \item
    $U(L) / U(L) A$ is isomorphic, as a filtrated $A$-module, to $S_R(L / A)$, where $U(L)$ is the universal enveloping algebra of $L$ and $S_R(L/A)$ is the sheaf of symmetric algebras generated by $L/A$.
\end{enumerate}
In our situation where $R = \Osheaf_X$, $L = TX_{\real} \otimes \complex$ and $A = T^{0,1} X$, the sheaf $U(L) / U(L) A$ in the third condition above is nothing but the universal enveloping algebra $U(T^{1,0}X) = U(TX)$ of the Lie algebroid $T^{1,0} X = TX$ of smooth $(1,0)$-vector fields, with the $A$-module structure being the one induced by the holomorphic structure. Since $\JetX \simeq \Hom_{\Osheaf_X} (U(TX), \Osheaf_X)$ and $\Shat(TX) \simeq \Hom_{\Osheaf_X}(S(TX), \Osheaf_X)$, the equivalence between conditions (1) and (3) is essentially the same as our Theorem \ref{thm:jet_Atiyah}. Also one can think formally of the first formal neighborhood $A^{(1)}$ in this case as the 'Koszul dual' of the somewhat naive dga $\Omega^{\leq 1}_X$.

\end{digression}

\section{Formal geometry}\label{sec:Formal}

In this section, we will review basics of formal geometry developed by Gelfand, Kazhdan and Fuchs (see \cite{GelfandKazhdan}, \cite{GelfandKazhdanFuks}). We will follow \cite{BeKaledin} and base our constructions of various jet bundles by applying Borel construction to the bundle $X_{coor}$ of formal coordinates (which also appears in the context of deformation quantization, see, e.g., \cite{Kontsevich}, \cite{BeKaledin}, \cite{Yekutieli}), thought as a torsor over some proalgebraic group $G\formal$. In particular, we obtain the bundle $X_{exp}$ of formal exponential maps used in \cite{Kapranov} and show that it can be naturally identified with a new defined bundle $X_{conn}$ of jets of holomorphic connections that are flat and torison-free. This reinterpretation of $X_{exp}$ gives a geometric meaning of the formula \eqref{eq:exp_Kahler} in \ref{subsec:Kapranov}, which is close to Kapranov's original approach, and allow us to work in general without the K\"ahler assumption.

We should mention that this section overlaps a lot with the corresponding discussions in \cite{Kapranov}, though we are working in the Dolbeault picture, and we claim no novelty except for the definition of $X_{conn}$.

\subsection{Differential geometry of formal discs}\label{subsec:Formal_discs}

Fix a complex vector space $V$ of dimension $n$. Consider the formal power series algebra
\begin{displaymath}
  \F = \complex \llbracket V^* \rrbracket = \Shat(V^*) = \prod_{i \geq 0} S^i V^*,
\end{displaymath}
that is, the function algebra of the formal neighborhood of $0$ in $V$. It is a complete regular local algebra with a unique maximal ideal $\mathfrak{m}$ consisting of formal power series with vanishing constant term. The associated graded algebra with respect to the $\mathfrak{m}$-filtration is the (uncompleted) symmetric algebra
\begin{displaymath}
  \graded \F = S(V^*) = \bigoplus_{i \geq 0} S^i V^*.
\end{displaymath}
Since we are in the complex analytic situation, we endow $\F$ with the canonical Fr\'echet topology . In algebraic setting, one need to use the $\mathfrak{m}$-adic topology on $\F$. However, the associated groups and spaces in question remain the same, though the topologies on them will be different. Since our arguments work for both Fr\'echet and $\mathfrak{m}$-adic settings, the topology will not be mentioned explicitly unless necessary. We also use $\widehat{V} = \Spf \F$ to denote the formal polydisc, either as a formal analytic space or a formal scheme.

Following the notations in \cite{Kapranov}, we denote by $G\formal = G\formal(V)$ the proalgebraic group of automorphisms of the formal space $\widehat{V}$, and by $J\formal = J\formal(V)$ the normal subgroup consisting of those $\phi \in G\formal$ with tangent map $d_0 \phi = \Id$ at $0$. In other words, $J\formal(V)$ is the kernel of $d_0 : G\formal \to GL_n(V)$. Let $\mathbf{g}\formal = \mathbf{g}\formal(V)$ and $\mathbf{j}\formal = \mathbf{j}\formal(V)$ be the corresponding Lie algebras. The Lie algebra $\mathbf{g}\formal$ can also be interpreted as the Lie algebra of formal vector fields vanishing at $0$, while $\mathbf{j}\formal$ is the Lie subalgebra of formal vector fields with vanishing constant and linear terms. We have decompositions
\begin{displaymath}
  \mathbf{g}\formal = \prod_{i \geq 1} V \otimes S^i V^* = \prod_{i \geq 1} \Hom (V^*, S^i V^*)
\end{displaymath}
and
\begin{displaymath}
  \mathbf{j}\formal = \prod_{i \geq 2} V \otimes S^i V^* = \prod_{i \geq 2} \Hom (V^*, S^i V^*).
\end{displaymath}
Elements of $\mathbf{g}\formal$ and $\mathbf{j}\formal$ acts on $\F = \prod_{i \geq 0} S^i V^*$ as derivations in the obvious manner.

There is an exact sequence of proalgebraic groups
\begin{equation}\label{exsq:progroups}
  1 \to J\formal \to G\formal \to \GL_n(V) \to 1
\end{equation}
which canonically splits if we regard elements of $\GL_n = \GL_n(V)$ as jets of linear transformations on $\widehat{V}$. In other words, $G\formal$ is a semidirect product:
\begin{displaymath}
  G\formal = J\formal \rtimes \GL_n.
\end{displaymath}
So we have a canonical bijection between sets
\begin{equation}\label{eq:J=G/GL}
  q: J\formal \xrightarrow{\simeq} G\formal / \GL_n.
\end{equation}
Moreover, there is a natural left $G\formal$-action on $J\formal$ given by
\begin{equation}\label{eq:actionJ_1}
  (\phi, T) \cdot \varphi = \phi \circ T \circ \varphi \circ T^{-1}, \quad \forall ~ (\phi, T) \in J\formal \rtimes \GL_n = G\formal, ~ \forall ~ \varphi \in J\formal,
\end{equation}
or equivalently,
\begin{equation}\label{eq:actionJ_2}
  \psi \cdot \varphi = \psi \circ \varphi \circ (d_0 \psi)^{-1}, \quad \forall ~ \psi \in G\formal, ~ \forall ~ \varphi \in J\formal,
\end{equation}
which makes $q$ into a $G\formal$-equivariant map. Moreover, an automorphism $\varphi$ in $J\formal$ can also be interpreted as a bijective morphism $\varphi: \widehat{T_0 V} \to \widehat{V}$, where $\widehat{T_0 V}$ is the completion of the tangent space $T_0 V$ at the origin, which is naturally identified with $\widehat{V}$ itself. In addition, $\varphi$ should induce the identity map on tangent spaces of two formal spaces at the origins (which are both equal to $V$). We call such a map $\varphi$ as a \emph{formal exponential map} since it satisfies analogous properties of exponential maps in classical Riemannian geometry. The above $G\formal$-action on $J\formal$ then has a clearer meaning: think of $\psi \in G\formal$ as a change of formal coordinates on $\widehat{V}$ and the image of $\varphi$ under this transformation is $\varphi$ composed with $\psi: \widehat{V} \to \widehat{V}$ and precomposed with the inverse of the linearization of $\psi$ on $\widehat{T_0 V}$. In other words, $J\formal$ is the set of all formal exponential maps $\varphi: \widehat{T_0 V} \to \widehat{V}$, on which the action of $G\formal$ comes from those on $\widehat{V}$ and $\widehat{T_0 V}$ with the latter being induced from the 'linearlization map' $G\formal \to \GL_n$.

To avoid confusion, we denote by
\begin{displaymath}
\F_T := \Shat ((T_0 V)^*) = \Shat (V^*)
\end{displaymath}
the algebra of functions on $\widehat{T_0 V}$ and by $\mathfrak{m}_T$ its maximal ideal, though it is just another copy of $\F$. Having a formal exponential map $\varphi: \widehat{T_0 V} \to \widehat{V}$ is the same as having an isomorphism of algebras
\begin{equation}\label{eq:exp_pullback}
  \varphi^*: \F \to \F_T
\end{equation}
whose induced isomorphism between the associated graded algebras, which are both equal to $S(V^*)$, is the identity. As before, we want to set the $G\formal$-action on $\F_T$ as the one induced from the linear action of $\GL_n$ on $\F_T$ and the projection $G\formal \to \GL_n$.

Now we give a another interpretation of $J\formal$ as a $G\formal$-space. Define $\Conn$ to be the space of all flat torsion-free connections on $\widehat{V}$,
\begin{displaymath}
  \nabla : T\widehat{V} \to T^*\widehat{V} \otimes_{\Osheaf_{\widehat{V}}} T\widehat{V} = V^* \otimes T\widehat{V},
\end{displaymath}
where $T\widehat{V} = \Shat(V^*) \otimes V$ and $T^*\widehat{V} = \Shat(V^*) \otimes V^*$ are (sections of) the tangent bundle and cotangent bundle of $\widehat{V}$ respectively. Most of time we also write $\nabla$ for the induced connections on the cotangent bundle and its tensor bundles, e.g.,
\begin{displaymath}
  \nabla : T^*\widehat{V} \to T^*\widehat{V} \otimes_{\Osheaf_{\widehat{V}}} T^*\widehat{V}.
\end{displaymath}
We identify $T^*\widehat{V} \otimes_{\Osheaf_{\widehat{V}}} T^*\widehat{V}$ with $\Shat(V^*) \otimes V^* \otimes V^*$, of the latter the first and second $V^*$ come from the first and second $T^*\widehat{V}$ of the former respectively. For any complex vector space $V$, there is a standard Euclidean connection $\EuclideanConn$ on $\widehat{V}$ (acting on $T^*\widehat{V})$
\begin{displaymath}
  \EuclideanConn: \prod_{n \geq 0} S^n V^* \otimes V^* \to \prod_{n \geq 1} S^{n-1} V^* \otimes V^* \otimes V^*,
\end{displaymath}
determined by its components
\begin{displaymath}
  (\EuclideanConn)_n: S^n V^* \otimes V^* \to S^{n-1} V^* \otimes V^* \otimes V^*
\end{displaymath}
defined by
\begin{displaymath}
  (\EuclideanConn)_n (v_1 v_2 \cdots v_n \otimes u) = \sum_{i=1}^{n} v_1 \cdots \hat{v}_i \cdots v_n \otimes v_i \otimes u.
\end{displaymath}
Using $\EuclideanConn$ we can rewrite the components of $f$ in the decomposition $\F = \prod_{i \geq 0} S^i V^*$ as
\begin{equation}\label{eq:euclid_conn}
  f = (\EuclideanConn^k f |_0)_{k \geq 0} = (f(0), \EuclideanConn f |_0, \EuclideanConn^2 f |_0, \cdots) \in \F,
\end{equation}
where $\EuclideanConn f = df \in T^*\widehat{V}$ is the usual differential of functions and $\EuclideanConn^k f = \EuclideanConn^{k-1} df$ ($k \geq 2$) while $|_0$ means restriction of the tensors at the origin.

As in classical differential geometry, for each connection $\nabla \in \Conn$ on the formal space $\widehat{V}$ we can assign an formal exponential map $\expmap_\nabla : \widehat{T_0 V} \to \widehat{V}$, which is defined to be the unique map in $J\formal$ such that it pulls back $\nabla$ into the Euclidean connection $\EuclideanConn$ on the completed tangent space, i.e.,
\begin{equation}\label{eq:exp_property}
  \expmap^*_\nabla \nabla = \EuclideanConn.
\end{equation}
To be more explicit, we define $\expmap_\nabla$ by defining its pullback map of functions $\expmap^*_\nabla: \F \to \F_T$:
\begin{equation}\label{eq:exp_definition}
  \expmap^*_\nabla (f) = (\nabla^i f |_0)_{i \geq 0} = (f(0), \nabla f |_0, \nabla^2 f |_0, \cdots) \in \prod_{i \geq 0} S^i (T_0 V)^* = \F_T
\end{equation}
which is analogous to \eqref{eq:euclid_conn}. Here again $\nabla f = df$ and $\nabla^i f = \nabla^{i-1} df$ for $i \geq 2$. Note that to make sure the terms in the expression lie in symmetric tensors one has to resort to the torsion-freeness and flatness of $\nabla$.

On the other hand, any $\phi \in J\formal$ pushs forward the Euclidean connection $\EuclideanConn$ on $\widehat{T_0 V}$ to a flat torsion-free connection $\nabla = \phi_* \EuclideanConn$ on $\widehat{V}$. Thus by the discussion above we obtain a bijection
\begin{equation}\label{eq:Conn=J}
  \exp : \Conn \xrightarrow{\simeq} J\formal.
\end{equation}
Note that $G\formal$ naturally acts on $\Conn$ from left by pushing forward connections via automorphisms of $\widehat{V}$. Recall that we also have a $G\formal$ action on $J\formal$ defined by \eqref{eq:actionJ_2}.

\begin{lemma}
  The map $\exp : \Conn \to J\formal$ defined above is $G\formal$-equivariant.
\end{lemma}

\begin{proof}
  One only needs to notice the following commutative diagram of spaces equipped with connections
  \begin{diagram}
    (\widehat{T_0 V}, \partial)            & \rTo^{d_0 \varphi}    & (\widehat{T_0 V}, \partial)   \\
    \dTo^{\expmap_\nabla}                  &                       & \dTo_{\expmap_{\varphi_* \nabla}}  \\
    (\widehat{V}, \nabla)                      & \rTo^{\varphi}        & (\widehat{V}, \varphi_* \nabla)
  \end{diagram}
  which corresponds exactly to \eqref{eq:actionJ_2}.
\end{proof}

All the arguments remain valid for $V^{(r)} = \Spf \F / \mathfrak{m}^r$, the $r$-th order formal neighborhood of $V$, where $\mathfrak{m}$ is the maximal ideal of the local algebra $\F$. One can also define $G^{(r)}$, $J^{(r)}$, etc., and state similar results about them.

\subsection{Bundle of formal coordinates and connections}\label{subsec:Formal_connection}

We now introduce the bundle of formal coordinate systems $p : X_{coor} \to X$ of a smooth complex manifold $X$. By definition (\S 4.4., \cite{Kapranov}), for $x \in X$ the fiber $X_{coord, x}$ is the space of infinite jets of biholomorphisms $\varphi: V \simeq \complex^n \to X$ with $\varphi(0) = x$. Hence $X_{coor}$ is a natural a holomorphic principal $G\formal$-bundle (or a $G\formal$-torsor).

We have a canonical isomorphism between sheaves of algebras
\begin{equation}\label{eq:coord}
  X_{coor} \times_X \mathcal{J}^\infty_X \simeq X_{coor} \times \F.
\end{equation}
over $X_{coor}$. Thus $X_{coor}$ can also be characterized by the following universal property (see \cite{BeKaledin}, \cite{Yekutieli}): given any complex space $S$, a morphism $\eta: S\to X$ and an isomorphism $\zeta : \eta^* \mathcal{J}^\infty_X \simeq \Osheaf_S \hat{\otimes} \F$ of sheaves of topological algebras over $S$, there is a unique morphism $\eta' : S \to X_{coor}$ such that $\eta = p \circ \eta'$ and $\zeta$ is induced from the isomorphism \eqref{eq:coord}. Note that such $\zeta$ does not necessarily exist for arbitrary $\eta: S \to X$, but always does for those $S$ which are Stein.

One can obtain various canonical jet bundles on $X$ by applying the associated bundle construction on principal $G\formal$-bundle $X_{coor}$. For example, $\F$ is naturally a left $G\formal$-module and the corresponding sheaf of algebras associated to the $G\formal$-torsor $X_{coor}$ coincides with the jet bundle of holomorphic functions $\mathcal{J}^\infty_X$:
\begin{displaymath}
  X_{coor} \times_{G\formal} \F \simeq \mathcal{J}^\infty_X
\end{displaymath}
To get the natural flat connection on $\mathcal{J}^\infty_X$, one could adopt the language of Harish-Chandra torsors as in \cite{BeKaledin}. We omit it since this connection will not be used in this paper. Other jet bundles, such as $\mathcal{J}^\infty T_X$, the jet bundle of the tangent bundle, and $\mathcal{J}^\infty T^*_X$, the jet bundle of cotangent bundle, can be obtained in a similar way. Again we refer interested readers to \cite{BeKaledin}.

Another related space which is at the very core of our discussions is the bundle of formal exponential maps introduced in \cite{Kapranov}, which we denote by $X_{exp}$ (some literatures use the notation $X_{aff}$). Each fiber $X_{exp, x}$ of $X_{exp}$ at $x \in X$ is the space of jets of holomorphic maps $\phi: T_x X \to X$ such that $\phi(0) = x$, $d_0 \phi = \Id$. We have a map
\begin{displaymath}
  X_{coor} \to X_{exp}, \quad \phi \mapsto \phi \circ (d_0 \phi)^{-1}
\end{displaymath}
which induce a biholomorphism
\begin{equation}\label{eq:GL=exp}
  X_{coor} / \GL_n \simeq X_{exp}
\end{equation}
On the other hand, we can define the bundle of jets of flat torsion-free connection
\begin{displaymath}
  X_{conn} = X_{coor} \times_{G\formal} \Conn
\end{displaymath}
whose fiber at a given point $x \in X$ consists of all flat torsion-free connections on the formal neighborhood of $x$. By combining the $G\formal$-equivariant bijections \eqref{eq:J=G/GL}, \eqref{eq:Conn=J} and \eqref{eq:GL=exp}, we get
\begin{equation}
  X_{conn} \simeq X_{coor} \times_{G\formal} J\formal \simeq X_{coor} \times_{G\formal} G\formal / \GL_n \simeq X_{exp}.
\end{equation}
In other words, we can naturally identify jet bundle of flat torsion-free connections with the jet bundle of formal exponential maps. From now on, we will not distinguish between these two jet bundles and denote both by $X_{conn}$, though both interpretations will be adopted in the rest of the paper.

\subsection{Tautological exponential map}\label{subsec:Formal_exp}

By definition of $X_{conn}$, there is a tautological flat and torsion-free connection over $X_{conn}$,
\begin{displaymath}
  \nabla_{tau} : \pi^* \mathcal{J}^\infty T^*X \to \pi^* \mathcal{J}^\infty T^*X \otimes_{\pi^* \mathcal{J}^\infty_X} \pi^* \mathcal{J}^\infty T^*X,
\end{displaymath}
which is $\Osheaf_{X_{coor}}$-linear yet satisfies the Leibniz rule with respect to the differential
\begin{displaymath}
  \widetilde{d}\formal: \pi^* \mathcal{J}^\infty_X \to \pi^* \mathcal{J}^\infty T^*X
\end{displaymath}
that is the pullback of
\begin{displaymath}
  d\formal: \mathcal{J}^\infty_X \to \mathcal{J}^\infty T^*X.
\end{displaymath}
Here $d\formal$ is the $\Osheaf_X$-linear differential obtained by applying the Borel construction with $X_{coor}$ and the differential $d : \Osheaf_{\widehat{V}} \to T^*\widehat{V}$ on the formal disc.

On the other hand, since $X_{conn}$ can also be interpreted as the bundle of formal exponential maps, we have a tautological isomorphism between sheaves of algebras over $X_{conn}$,
\begin{equation}\label{eq:Exp_definition}
  \Expmap^*: \pi^* (X_{coor} \times_{G\formal} \F) \to \pi^* (X_{coor} \times_{G\formal} \F_T),
\end{equation}
induced from $\expmap^*$ in \eqref{eq:exp_definition} via the Borel construction. The domain of $\Expmap^*$ is identified with $\pi^* \JetX$ or $\pi^* \Osheaf_{\Xdiag}$, while for the codomain we have
\begin{displaymath}
  X_{coor} \times_{G\formal} \F_T \simeq X_{coor} \times_{G\formal} \GL_n \times_{\GL_n} \F_T \simeq X_{coor} / J\formal \times_{\GL_n} \F_T
\end{displaymath}
by our definition of the $G\formal$-action on $\F_T$. But the principal $\GL_n$-bundle $X_{coor} / J\formal$ is exactly the bundle of ($0$th-order) frames on $X$, so
\begin{displaymath}
  X_{coor} / J\formal \times_{\GL_n} V \simeq TX.
\end{displaymath}
Since the $\GL_n$ action respects the decomposition $\F_T = \prod_{i \geq 0} S^i V^*$, we get
\begin{displaymath}
  X_{coor} / J\formal \times_{\GL_n} \F_T \simeq \prod_{i \geq 0} S^i T^*X = \Shat(T^*X),
\end{displaymath}
which is the structure sheaf of $X\formal_{TX}$, the formal neighborhood of the zero section of $TX$. In short, we have a \emph{tautological exponential map}
\begin{displaymath}
  \Expmap: X_{conn} \times_X X\formal_{TX} \to X_{conn} \times_X \Xdiag
\end{displaymath}
or equivalently, an isomorphism of bundles of topological algebras
\begin{displaymath}
  \Expmap^*: \pi^* \Osheaf_{\Xdiag} \to \pi^* \Osheaf_{X\formal_{TX}}.
\end{displaymath}
which we might call as the \emph{tautological Taylor expansion map}. Moreover, the map induced by $\Expmap^*$ between associated bundle of graded algebras
\begin{displaymath}
  \graded \Expmap^* : \pi^* \graded \Osheaf_{\Xdiag} = \pi^* S(T^*X) \to \pi^* S(T^*X)
\end{displaymath}
is the identity map. In virtue of \eqref{eq:exp_definition}, $\Expmap^*$ can be written in terms of $\nabla_{tau}$:
\begin{equation}
  \Expmap^* (f) = (\nabla_{tau}^i f |_0)_{i \geq 0} = (f(0), \nabla_{tau} f |_0, \nabla^2_{tau} f |_0, \cdots) \in \pi^* \Shat(T^*X)
\end{equation}
where the 'restriction to the origin' map $\pi^* S^i \mathcal{J}^\infty T^*X \to \pi^* S^i T^*X$ comes from the local restriction map $T^*\widehat{V} \to T^*_0 \widehat{V} = V^*$ by applying Borel construction with $X_{coor}$ and then pulling back to $X_{conn}$ via $\pi$. Again $\nabla_{tau} f$ means $d\formal f$ and so on.

\begin{remark}\label{rmk:Jformal}
Note that there is no natural $G\formal$- or $J\formal$-action on $X_{exp}=X_{coor}$. Yet it is a torsor over the proalgebraic group bundle $J\formal(TX)$, whose fiber over $x \in X$ is the group of jets of biholomorphisms $\varphi: T_x X \to  T_x X$ with $\varphi(0)=0$, $d_0\varphi = \Id$. Indeed, consider the proalgebraic group $\Aut_0 \widehat{T_0 V}$ of automorphisms of $\widehat{T_0 V}$ whose tangent maps are the identity. $\Aut_0 \widehat{T_0 V}$ can be identified with $J\formal$ as sets, yet we endow it with a different $G\formal$-action which is the conjugation of the one on $\widehat{T_0 V}$ induced from $G\formal \to \GL_n$. $\Aut_0 \widehat{T_0 V}$ acts on $J\formal$ from right by precomposition and the action is compatible with the $G\formal$-actions. Finally, notice that
\begin{displaymath}
  J\formal(TX) \simeq X_{coor} \times_{G\formal} \Aut_0 \widehat{T_0 V}.
\end{displaymath}
Let $\mathbf{j}\formal(TX)$ be the bundle of Lie algebras associated to $J\formal(TX)$. We have a natural splitting
\begin{equation}\label{eq:j_comp}
  \mathbf{j}\formal(TX) = \prod_{i \geq 2} \Hom(S^i TX,TX) = \prod_{i \geq 2} \Hom(T^*X, S^i T^*X).
\end{equation}
\end{remark}

\begin{remark}
All the jet bundles we are discussing here are holomorphic and although in general they might not admit global holomorphic sections, there always exist global smooth sections. For example, let us consider the fiber bundle $\pi_n: X_{exp}^{(n)}(X) \to X$ of `$n$-th order exponential maps', cf. \S 4.2., \cite{Kapranov}. By definition, for each $x \in X$ the fiber $X_{exp, x}^{(n)}$ is the space of $n$-th order jets of holomorphic maps $\phi : T_x X \to X$ such that $\phi(0) = x$, $d_0 \phi = \Id$. Thus we we have a chain of projections
\begin{equation}\label{diag:Phi}
  X \leftarrow X_{exp}^{(2)} \leftarrow X_{exp}^{(3)} \leftarrow \cdots.
\end{equation}
Each $X_{exp}^{(n+1)}$ is an affine bundle over $X_{exp}^{(n)}$ whose associated vector bundle is $$\pi_n^* \Hom(S^{n+1}TX, TX),$$ so any section of $X_{exp}^{(n)}$ can be lifted to a smooth section of the next bundle in the diagram. The inverse limit of the diagram is exactly the bundle $\pi: X_{exp} \to X$. Thus $X_{exp}$ admits global smooth sections. Note that in picture of connections, $X_{exp}^{(2)} = X_{conn}^{(2)}$ can be regarded as the bundle of torsion-free connections (see Section 2.2., \cite{Kapranov}).
\end{remark}

\subsection{Kapranov's theorem revisited}\label{subsec:Kapranov_revisited}

We now show how to implement the formal analysis from previous sections to prove Kapranov's theorem, in a more general form. Given any smooth section $\sigma$ of $X_{conn}$, it induces a smooth homomorphism via the tautological Taylor expansion map $\Expmap^*$:
\begin{displaymath}
  \expmap^*_{\sigma} : \JetX \to \Shat(T^*X),
\end{displaymath}
between Fr\'echet bundles of algebras over $X$. It is holomorphic if and only if $\sigma$ is holomorphic. Conversely, $X_{conn}$ satisfies the universal property similar to that of $X_{coor}$. Namely, given a holomorphic (resp. smooth) map $\eta: S \to X$, any holomorphic (resp. smooth) isomorphism $\zeta : \eta^* \JetX \to \eta^* \Shat(T^*X)$ of sheaves of $\Osheaf_S$-modules, which induces the identity map on the associated graded algebras
\begin{displaymath}
  \graded \JetX = S(T^*X) = \graded \Shat(T^*X),
\end{displaymath}
arises from $\Expmap^*$ and a unique holomorphic (resp. smooth) section $\eta': S \to X_{conn}$ such that $\eta = \pi \circ \eta'$.

In particular, global smooth sections of $X_{conn}$ correspond in a $1-1$ manner to all possible smooth isomorphisms between $\JetX$ and $\Shat(T^*X)$ which induce the identity map on the associated graded algebras. As a holomorphic principle $J\formal(TX)$-bundle (see Remark \ref{rmk:Jformal}), $X_{conn}$ carries a flat $(0,1)$-connection $\overline{d}$, such that for any given smooth section $\sigma$ of $X_{conn}$, its anti-holomorphic differential
\begin{equation}\label{eq:omega}
  \omega_\sigma := \overline{d} \sigma \in \A^{0,1}(\mathbf{j}\formal(TX))
\end{equation}
is well-defined and satisfies the Maurer-Cartan equation
\begin{equation}\label{eq:MCeqn_omega}
  \dbar \omega_\sigma + \frac{1}{2} [\omega_\sigma, \omega_\sigma] = 0.
\end{equation}
Decompose $\omega$ as in \eqref{eq:j_comp}, we get $(0,1)$-forms
\begin{displaymath}
  \alpha_\sigma^n \in \A^{0,1}_X(\Hom(S^n TX, TX)) = \A^{0,1}_X(\Hom(T^*X, S^n T^*X)).
\end{displaymath}
Moreover, one can extend the map $\expmap^*_\sigma$ linearly with respect to the $\AOD(X)$-actions to a homomorphism between graded algebras
\begin{equation}
  \expmap^*_\sigma : \AOD_X(\JetX) \to \AOD_X(\Shat(T^*X)).
\end{equation}
In light of Proposition \ref{prop:isom_diag_dga}, we abuse our notations and still use $\expmap^*_\sigma$ for the composition $\expmap^*_\sigma \circ (I_\infty)^{-1}$, so that we have
\begin{equation}\label{eq:expmap}
  \expmap^*_\sigma : \A^\bullet(X\formal_{X \times X}) \to \AOD_X(\Shat(T^*X)).
\end{equation}
However, the map \eqref{eq:expmap} does not need to commute with the $\dbar$-differentials on both sides. The deficiency is measured exactly by $\omega_\sigma$, that is,
\begin{equation}\label{eq:omega_exp}
  \omega_\sigma = (\dbar \expmap^*_\sigma) \circ (\expmap^*_\sigma)^{-1},
\end{equation}
where $\dbar \expmap^*_\sigma = [\dbar, \expmap^*_\sigma]$. This suggests that we can correct the usual holomorphic structure on $\Shat(T^*X)$ to make $\expmap^*_\sigma$ into a map of dgas. Let $\widetilde{\omega}$ and $\widetilde{\alpha}_\sigma^n$ be the odd derivations of the graded algebra $\AOD(\Shat(T^*X))$ induced by $\omega$ and $\alpha_\sigma^n$ respectively. Define a new differential $D_\sigma = \dbar - \widetilde{\omega} = \dbar - \sum_{n \geq 2} \widetilde{\alpha}_\sigma^n$, then $D^2_{\sigma} = 0$, which is equivalent to  then the Maurer-Cartan equation \eqref{eq:MCeqn_omega}. We have the following generalization of Theorem \ref{thm:Kapranov_Diagonal}:

\begin{proposition}\label{prop:exponential_map}
The Taylor expansion map with respect to any given smooth section $\sigma$ of $X_{conn}$
\begin{displaymath}
  \expmap^*_\sigma : (\A^\bullet(X\formal_{X \times X}),\dbar) \to (\AOD(\Shat(T^*X)), D_\sigma)
\end{displaymath}
is an isomorphism of dgas. 
\end{proposition}

If we denote the value of $\sigma$ at each point $x \in X$ as $\nabla_x$, where $\nabla_x$ is a flat torsion-free connection of $T^*X$ on the formal neighborhood of $x$ in $X$, then the map $\expmap^*_\sigma$ regarded as a homomorphism between sheaves of smooth sections of $\JetX$ and $\Shat(T^*X)$ can be formally expressed as
\begin{equation}\label{eq:exp_formal}
  \expmap^* ([f]_\infty) |_x= (f |_x, (\partial f) |_x, (\nabla^2_x f) |_x, \cdots, (\nabla^n_x f)|_x, \cdots),
\end{equation}
where $|_x$ means taking the value at $x$ of the formal tensors fields on the formal neighborhood of $x$. We can also write down formula for each $\alpha^n_\sigma$ of similar form as that of $R_n$ in Theorem \ref{thm:Kapranov_Diagonal}. Namely, if we regard $\sigma$ as a section of the affine bundle $X^{(2)}_{conn}$ via the projection $X_{conn} \to X^{(2)}_{conn}$, then $\alpha_2 = \dbar \sigma$ and the values of $\alpha^n$ $(n > 2)$ at each point $x$ is given by
\begin{displaymath}
  \alpha_n |_x = (\nabla^{n-2}_x \alpha_2) |_x.
\end{displaymath}

\begin{remark}
  If $X$ is K\"ahler, we can take $\sigma$ to be the holomorphic jets of the canonical $(1,0)$-connection $\nabla$. The resulting map $\expmap^*_\sigma$ from \eqref{eq:exp_formal} is exactly the pullback map of Kapranov's exponential map. One should compare \eqref{eq:exp_formal} with \eqref{eq:exp_Kahler} which we used to prove Theorem \ref{thm:Kapranov_Diagonal} and notice that the symbols $\nabla$ in these two formulae have \emph{different} meanings! The $\nabla$ in the former means the holomorphic jets of the actual connection, which are holomorphic on the formal neighborhood of each point $x \in X$, while in the latter the actual connection is used but it is not holomorphic. In fact, however, \eqref{eq:exp_Kahler} only cares about the holomorphic jets of the actual connection $\nabla$, since it is a $(1,0)$-connection! So the two pullback maps $\expmap^*$ \emph{are} the same, yet not for the wrong reason that they look like the same.
\end{remark}

\bibliographystyle{halpha}
\bibliography{GeometryFormal_I_bib}

\end{document}